\documentclass[amstex,12pt,russian,amssymb]{article}

\usepackage{mathtext}
\usepackage[cp1251]{inputenc}
\usepackage[T2A]{fontenc}
\usepackage[russian]{babel}
\usepackage[dvips]{graphicx}
\usepackage{amsmath}
\usepackage{amssymb}
\usepackage{amsxtra}
\usepackage{latexsym}
\usepackage{ifthen}

\begin{document}

\newcounter{lemma}
\newcommand{\lemma}{\par \refstepcounter{lemma}%
{\bf Лема \arabic{lemma}.}}

\newcounter{corollary}
\newcommand{\corollary}{\par \refstepcounter{corollary}%
{\bf Наслідок \arabic{corollary}.}}

\newcounter{remark}
\newcommand{\remark}{\par \refstepcounter{remark}%
{\bf Зауваження \arabic{remark}.}}

\newcounter{theorem}
\newcommand{\theorem}{\par \refstepcounter{theorem}%
{\bf Теорема \arabic{theorem}.}}

\newcounter{proposition}
\newcommand{\proposition}{\par \refstepcounter{proposition}%
{\bf Твердження \arabic{proposition}.}}

\newcounter{example}
\newcommand{\example}{\par \refstepcounter{example}%
{\bf Приклад \arabic{example}.}}

\renewcommand{\refname}{\centerline{\bf Список літератури}}

\renewcommand{\figurename}{Мал.}

\newcommand{\proof}{{\it Доведення.\,\,}}

\noindent УДК 517.5

\medskip\medskip
{\bf Є.О.~Севостьянов} (Житомирський державний університет імені
Івана Фран\-ка; Інститут прикладної математики і механіки НАН
України, м.~Слов'янськ)

{\bf В.А.~Таргонський, Н.С.~Ількевич} (Житомирський державний
університет імені Івана Фран\-ка)

\medskip\medskip
{\bf E.O.~Sevost'yanov} (Zhytomyr Ivan Franko State University;
Institute of Applied Ma\-the\-ma\-tics and Mechanics of NAS of
Ukraine, Slov'yans'k)

{\bf V.A.~Targonskii, N.S.~Ilkevych} (Zhytomyr Ivan Franko State
University)

\medskip
{\bf Про квазілінійні рівняння Бельтрамі з обмеженнями на дотичну
дилатацію}

{\bf On quasilinear Beltrami equations with restrictions on
tangential dilatation}

\medskip\medskip
Досліджуються квазілінійні рівняння Бельтрамі, комплексні
коефіцієнти котрих залежать від невідомої функції. У термінах так
званої дотичної дилатації знайдені умови, за яких ці рівняння мають
гомеоморфні $ACL$-розв'язки. Окремо знайдені умови, які забезпечують
існування відповідних неперервних $ACL$-розв'язків.

\medskip\medskip
We study quasilinear Beltrami equations, the complex coefficients of
which depend on the unknown function. In terms of the so-called
tangential dilatation, we have found conditions under which these
equations have homeomorphic $ACL$-solutions. Sepa\-ra\-tely, we have
found some conditions that ensure the existence of the corresponding
con\-ti\-nuous $ACL$-solutions.

\newpage
{\bf 1. Вступ.} Теореми існування розв'язків рівнянь Бельтрамі
займають важливе місце в дослідженнях з сучасного аналізу. Роботи на
цю тему активно публікуються різними авторами у всьому світі, див.,
напр., \cite{A}--\cite{SalSt}. Окремо відзначимо відповідні
публікації донецької і житомирської наукових шкіл, див., напр.,
\cite{GRSY}, \cite{MRSY},  \cite{RSY$_2$}, \cite{Sev$_1$} і
\cite{SevSkv}. Зокрема, в роботах \cite{DS$_1$} і \cite{DS$_2$}
першого співавтора досліджені вироджені рівняння Бельтрамі як з
одною, так і двома комплексними характеристиками $\mu(z, f)$ і
$\nu(z, f),$ які можуть залежати від самого розв'язку $f.$ За
відповідних умов на так звану максимальну дилатацію рівняння були
отримані теореми про існування як гомеоморфних, так і неперервних
розв'язків у одиничному крузі. Мета даного рукопису -- посилити ці
результати на випадок, коли рівняння розглядається в довільній
області $D\subset {\Bbb C},$ а умови формулюються в термінах так
званої тангенсальної дилатації.

\medskip
Нагадаємо деякі означення, див.~\cite{GRSY}, \cite{RSY$_2$}
і~\cite{DS$_1$}. Скрізь далі ми припускаємо, що відображення
$f:D\rightarrow {\Bbb C}$ області $D\subset{\Bbb C}$ {\it зберігає
орієнтацію,} тобто, якщо $f$ -- відкрите дискретне відображення і
$z\in D$ -- яка-небудь його точка диференційовності, то {\it
якобіан} цього відображення в точці $z$ невід'ємний (див., напр.,
\cite[лема~2.14]{MRV$_1$}). Для комплекснозначної функції
$f:D\rightarrow {\Bbb C},$ заданої в області $D\subset {\Bbb C},$ що
має частинні похідні по $x$ і $y$ при майже всіх $z=x + iy,$
покладемо $f_{\overline{z}} = (f_x + if_y)/2$ і $f_z = (f_x -
if_y)/2.$ Будемо говорити, що функція $\nu=\nu(z, w):D\times{{\Bbb
C}}\rightarrow {\Bbb D}$ задовольняє {\it умову Каратеодорі,} якщо
$\nu$ вимірна по $z\in D$ при кожному фіксованому $w\in{\Bbb C}$ і
неперервна по $w$ при майже всіх $z\in D.$

\medskip
Нехай функції $\mu=\mu(z, w)$ і $\nu=\nu(z, w)$ задовольняють умову
Каратеодорі і, крім того, $|\mu(z, w)|+|\nu(z, w)|<1$ при всіх $w\in
{\Bbb C}$ і майже всіх $z\in D.$ {\it Максимальною дилатацією,
побудованою по $\mu$ і $\nu$} називається наступна функція:
\begin{equation}\label{eq1}
K_{\mu, \nu}(z, w)=\quad\frac{1+|\mu(z, w)|+|\nu(z, w)|}{1-|\mu\,(z,
w)|-|\nu(z, w)|}\,.
\end{equation}
Згідно~\cite[розділ~11]{MRSY}, {\it тангенсальною (дотичною)}
дилатацією, побудованою по функціям $\mu(z, w)$ та $\nu(z, w)$
відносно точки $z_0\in {\Bbb C}$ та дійсного параметру $\theta\in
[0, 2\pi)$ будемо називати величину
\begin{equation}\label{eq20} K^T_{\mu, \nu}(z, z_0, w, \theta)=
\frac{\left|1-\frac{\overline{z-z_0}}{z-z_0}\left(\mu(z, w)+\nu(z,
w)e^{i\theta}\right)\right|^2}{1-|\mu(z, w)+\nu(z,
w)e^{i\theta}|^2}\,,
\end{equation}
-- де $z\in D$ -- точка невиродженої диференційовності $f.$
Зауважимо, що якобіан відображення $f$ в точці $z\in D$ можна
порахувати за допомогою рівності
$$J(z,
f)=|f_z|^2-|f_{\overline{z}}|^2\,.$$
Нехай функції $\mu=\mu(z, w)$ і $\nu=\nu(z, w)$ задовольняють умову
Каратеодорі і, крім того, $|\mu(z, w)|+|\nu(z, w)|<1$ при всіх $w\in
{\Bbb C}$ і майже всіх $z\in D.$ Розглянемо {\it квазілінійне
рівняння Бельтрамі з двома характеристиками:}
\begin{equation}\label{eq1:}
f_{\overline{z}}=\mu(z, f(z))\cdot f_z+\nu(z, f(z))\cdot
\overline{f_z}\,.
\end{equation}
Відображення $f:D\rightarrow {\Bbb C}$ будемо називати {\it
регулярним розв'язком рівняння~(\ref{eq1:}),} якщо $f\in W_{\rm
loc}^{1,1}$ і $J(z, f)\ne 0$ майже скрізь у $D.$ Будемо говорити, що
локально інтегровна в околі точки $x_0\in D$ функція
${\varphi}:D\rightarrow{\Bbb R}$ має {\it скінченне середнє
коливання} в точці $x_0$ (пишемо: $\varphi\in FMO(x_0)$), якщо
\begin{equation}\label{eq17:}
{\limsup\limits_{\varepsilon\rightarrow
0}}\frac{1}{\Omega_n\varepsilon^n}\int\limits_{B(
x_0,\,\varepsilon)}
|{\varphi}(x)-\overline{{\varphi}}_{\varepsilon}|\ dm(x)\, <\,
\infty\,,
\end{equation}
де $\Omega_n$ -- об'єм одиничної кулі в ${\Bbb R}^n,$
$\overline{{\varphi}}_{\varepsilon}=\frac{1}{\Omega_n\varepsilon^n}\int\limits_{B(
x_0,\,\varepsilon)} {\varphi}(x)\ dm(x)$ (див., напр.,
\cite[розд.~2]{RSY$_2$}). Позначимо через $q_{z_0}(r)$ середнє
значення функції $Q^1_{z_0}$ над колом $$S(z_0,
r)=\{|z-z_0|=r\}\,,$$
\begin{equation}\label{eq2}
q^1_{z_0}(r)=\frac{1}{2\pi}\int\limits_{0}^{2\pi}Q^1_{z_0}
(z_0+re^{\,i\varphi})\,d\varphi\,.
\end{equation}
Справедливе наступне твердження, справедливість якого для лінійних і
квазілінійних рівнянь встановлена в~\cite{DS$_1$}.

\medskip
\begin{theorem}\label{th4.6.1}{\sl\,
Нехай $D$ -- область у ${\Bbb C},$ і нехай функції $\mu=\mu(z, w)$ і
$\nu=\nu(z, w)$ задовольняють умову Каратеодорі і, крім того,
$|\mu(z, w)|+|\nu(z, w)|<1$ при всіх $w\in {\Bbb C}$ і майже всіх
$z\in D.$ Припустимо, що існує функція $Q:D\rightarrow [1, \infty]$
така, що $K_{\mu, \nu}(z, w)\leqslant Q(z)\in L_{\rm loc}^1(D)$ для
майже всіх $z\in D$ і всіх $w\in {\Bbb C},$ де функція $K_{\mu,
\nu}$ визначена в (\ref{eq1}). Припустимо, що для будь-якого $z_0\in
D$ існує функція $Q^1_{z_0}:D\rightarrow [0, \infty]$ така, що майже
всіх $z\in D,$ всіх $w\in {\Bbb C}$ і всіх $\theta\in [0, 2\pi)$
виконується нерівність
\begin{equation}\label{eq2A}
K^T_{\mu, \nu}(z, z_0, w, e^{i\theta})\leqslant Q^1_{z_0}(z)\,.
\end{equation}
Припустимо, що виконана одна з двох умов: або $Q^1_{z_0}\in FMO(D),$
або для кожного $z_0\in D$
\begin{equation}\label{eq15:}
\int\limits_{0}^{\delta(z_0)}\frac{dr}{rq^1_{z_0}(r)}=\infty\,,
\end{equation}
де $\delta(z_0)>0$ -- деяке число,
$\delta(z_0)<{\rm dist\,}(z_0,\partial D),$ а $q^1_{z_0}(r)$
визначено в~(\ref{eq2}).
Тоді рівняння (\ref{eq1:}) має регулярний гомеоморфний розв'язок $f$
класу $W_{\rm loc}^{1,1}$ в $D,$ такий що $f^{\,-1}\in W_{\rm
loc}^{\,1,2}(f(D)).$ Більше того, $f$ має гомеоморфне продовження в
${\Bbb C},$ яке є конформним зовні області $D,$ та його можна обрати
таким, що $f(0)=0,$ $f(1)=1.$ }
\end{theorem}

\medskip
Формулювання теореми~\ref{th4.6.1} зроблено в достатньо загальній
формі. Сформулюємо тепер декілька важливих наслідків, які
відносяться до цієї теореми. Передусім, прямими обчисленнями ми
отримаємо, що
$$K^T_{\mu, \nu}(z, z_0, w, \theta)\leqslant K_{\mu, \nu}(z, w)$$
для всяких $z_0\in D,$ майже всіх $z\in D,$ всіх $w\in {\Bbb C}$ і
всіх $\theta\in[0, 2\pi).$ Отже, справедливо наступне твердження.

\medskip
\begin{corollary}\label{cor1}
{\sl Твердження теореми~\ref{th4.6.1} виконується, якщо
умову~(\ref{eq2A}) замінити однією з умов: $Q\in FMO(D),$ або
\begin{equation}\label{eq3B}
\int\limits_{0}^{\delta(z_0)}\frac{dr}{rq_{z_0}(r)}=\infty\,,
\end{equation}
де $\delta(z_0)>0$ -- деяке число,
$\delta(z_0)<{\rm dist\,}(z_0,\partial D),$ а $q_{z_0}(r)$ визначено
співвідношенням

\begin{equation}\label{eq2D}
q_{z_0}(r)=\frac{1}{2\pi}\int\limits_{0}^{2\pi}Q(z_0+re^{\,i\varphi})\,d\varphi\,.
\end{equation}
}
\end{corollary}

Для формулювань подальших результатів і їх доведення нам необхідно
окремо окреслити випадок, коли максимальна і тангенсальна дилатації
залежать тільки від одної змінної $z.$ Нехай надалі функції $K_{\mu,
\nu}(z)$ і $K^T_{\mu, \nu}(z)$ визначені рівностями
$$K_{\mu,\nu}(z):=K_{\mu, \nu}(z, 0)\,,\qquad
K^T_{\mu,\nu}(z, z_0):=K^T_{\mu, \nu}(z, z_0, 0)\,,$$
крім того, нехай функції $K_{\mu}$ і $K^T_{\mu}$ визначені
рівностями
\begin{equation}\label{eq2E}
K_{\mu}(z):=K_{\mu, 0}(z, 0)\,,\qquad K^T_{\mu}(z, z_0):=K^T_{\mu,
0}(z, z_0, 0)\,.
\end{equation}
Поклавши в теоремі~\ref{th4.6.1} $\nu\equiv 0,$ маємо наступне.

\medskip
\begin{corollary}\label{cor2}{\sl\,
Нехай функція $\mu=\mu(z, w)$ задовольняє умову Каратеодорі і, крім
того, $|\mu(z, w)|<1$ при всіх $w\in {\Bbb C}$ і майже всіх $z\in
D.$ Припустимо, що існує функція $Q: D\rightarrow [1, \infty]$ така,
що $K_{\mu}(z, w)\leqslant Q(z)\in L_{\rm loc}^1(D)$ для майже всіх
$z\in D$ і всіх $w\in {\Bbb C}.$ Припустимо, що для кожного $z_0\in
D$ існує функція $Q^1=Q^1_{z_0}:D\rightarrow [0, \infty]$ така, що
для майже всіх $z\in D,$ всіх $w\in {\Bbb C}$ і всіх $\theta\in[0,
2\pi)$ виконується нерівність
\begin{equation}\label{eq4B}
K^T_{\mu}(z, z_0, w, \theta)\leqslant Q^1_{z_0}(z)\,,
\end{equation}
де функція $K^T_{\mu}(z, z_0, w, \theta)$ визначена в~(\ref{eq20}).
Припустимо, що виконана одна з двох умов: або $Q^1_{z_0}\in FMO(D),$
або для кожного $z_0\in D$ виконана рівність~(\ref{eq15:}), де
$\delta(z_0)>0$ -- деяке число, $\delta(z_0)<{\rm
dist\,}(z_0,\partial D),$ а $q^1_{z_0}(r)$ визначено в~(\ref{eq2}).
Тоді рівняння (\ref{eq1:}) має регулярний гомеоморфний розв'язок $f$
класу $W_{\rm loc}^{1,1}$ в $D,$ такий що $f^{\,-1}\in W_{\rm
loc}^{\,1,2}(f(D)).$ Більше того, $f$ має гомеоморфне продовження в
${\Bbb C},$ яке є конформним зовні області $D,$ та його можна обрати
таким, що $f(0)=0,$ $f(1)=1.$ }
\end{corollary}

\medskip
Оскільки $K^T_{\mu}(z, z_0)\leqslant K_{\mu}(z),$ з
наслідку~\ref{cor2} миттєво отримаємо наступне.

\medskip
\begin{corollary}\label{cor3}
{\sl Твердження наслідку~\ref{cor2} виконується, якщо
умову~(\ref{eq4B}) замінити однією з умов: $Q\in FMO,$
або~(\ref{eq3B}), де $\delta(z_0)>0$ -- деяке число,
$\delta(z_0)<{\rm dist\,}(z_0,\partial D),$ а $q_{z_0}(r)$ визначено
співвідношенням~(\ref{eq2D}).
}
\end{corollary}

\medskip
{\bf 2. Лема про існування гомеоморфного розв'язку лінійного
рівняння Бельтрамі з двома характеристиками}. Наступна лема була
доведена у~\cite[лема~9.1]{GRSY} для рівняння Бельтрамі з двома
характеристикою за відповідних умов на максимальну дилатацію
$K_{\mu, \nu}.$ Зараз ми сформулюємо і доведемо відповідне
твердження для тангенсальної дилатації $K^T_{\mu, \nu}.$

\medskip
\begin{lemma}\label{lem2}{\sl\,
Нехай $D$ -- область комплексної площини, і нехай функції
$\mu:D\rightarrow {\Bbb D}$ і $\nu: D\rightarrow {\Bbb D}$ є
вимірними за Лебегом і, крім того, $|\mu(z)|+|\nu(z)|<1$ при майже
всіх $z\in D.$ Нехай, крім того, існує функція $Q:D\rightarrow [1,
\infty]$ така, що $K_{\mu, \nu}(z)\leqslant Q(z)\in L_{\rm
loc}^1(D)$ для майже всіх $z\in D.$ Припустимо, що для будь-якого
$z_0\in D$ існує функція $Q^{z_0}_1(z):D\rightarrow [1, \infty)$
така, що $K^T_{\mu, \nu, \theta}(z, z_0)\leqslant Q^{z_0}_1(z)$ при
майже всіх $z\in D$ і всіх $\theta\in [0, 2\pi);$ крім того,
знайдуться $0<\varepsilon^{\,\prime}_0\leqslant \varepsilon_0< {\rm
dist\,}(z_0,
\partial D)$ і вимірна за Лебегом функція $\psi_{z_0}(0, \infty)\rightarrow (0,
\infty)$ такі, що
\begin{equation}\label{eq1E} 0<I(\varepsilon,
\varepsilon_0):=\int\limits_{\varepsilon}^{\varepsilon_0}\psi_{z_0}(t)\,dt
< \infty
\end{equation}
при $\varepsilon \in(0, \varepsilon^{\,\prime}_0),$ при цьому,
\begin{equation}\label{eq1F}
\int\limits_{\varepsilon<|z-z_0|<\varepsilon_0}
Q^1_{z_0}(z)\cdot\psi_{z_0}^{\,2}(|z-z_0|)
 \, dm(z)= o(I^{\,2}(\varepsilon, \varepsilon_0))\,.
\end{equation}
Тоді рівняння
\begin{equation}\label{eq1G}
f_{\overline{z}}=\mu(z)\cdot f_z+\nu(z)\cdot \overline{f_z}
\end{equation}
має регулярний гомеоморфний розв'язок $f$ класу $W_{\rm loc}^{1,1}$
в $D,$ такий що $f^{\,-1}\in W_{\rm loc}^{\,1,2}(f(D)).$}
\end{lemma}

\medskip
\begin{proof} Переважно будемо користуватися схемою
доведення леми~9.1 в~\cite{GRSY}. Покладемо $D_n:=D\cap B(0, n)$ і
розглянемо послідовності функцій
\begin{equation}\label{eq1H} \mu_n(z)= \left
\{\begin{array}{rr}
 \mu(z),& z\in B(0, n)\,, Q(z)\leqslant n,
\\ 0\ , & z\in B(0, n)\,, Q(z)>n
\end{array} \right.
\end{equation}
і
\begin{equation}\label{eq1I} \nu_n(z)= \left
\{\begin{array}{rr}
 \nu(z),& z\in B(0, n)\,, Q(z)\leqslant n,
\\ 0\ , & z\in B(0, n)\,, Q(z)>n\,.
\end{array} \right.
\end{equation}
Зауважимо, що $K_{\mu_n, \nu_n}(z)\leqslant n$ при майже всіх $z\in
D$ і всіх $w\in{\Bbb C}.$ Отже, за \cite[теорема~5.1]{Bo} існує
$n$-ква\-зікон\-фор\-мний розв'язок $\omega_n$ рівняння
$$\overline{\partial}\omega_n=\partial\omega_n\mu^{\,*}_n(z)+
\overline{\partial \omega_n}\nu^*_n(z)\,,$$
який квазіконформно відображає одиничний круг на себе такий, що
$\omega_n(0)=0,$ де $\mu^{\,*}_n(z)=\mu_n(nz)$ і
$\nu^{\,*}_n(z)=\nu_n(nz),$ $z\in {\Bbb D}.$ Зауважимо, що
$f_n(z)=\omega(z/n)/|\omega(1/n)|$ відображає одиничний круг на круг
$B(0, R_n),$ де $R_n=1/|\omega(1/n)|\geqslant 1,$ $f_n(0)=1,$
$|f_n(1)|=1$ і, крім того, $f_n$ задовольняє рівняння
\begin{equation}\label{eq1M}\overline{\partial}f_n=\partial f_n\mu_n(z)+
\overline{\partial f_n}\nu_n(z)\qquad \text{для\quad м.в.}\quad z\in
B(0, n)\,.
\end{equation}
За допомогою інверсії продовжимо відображення $f_n$ за межу
одиничного круга на розширену комплексну площину $\overline{\Bbb
C},$ так що $f_n(\infty)=\infty.$

\medskip
З іншого боку, з огляду на теорему~11.1 в~\cite{MRSY} кожне $f_n$
задовольняє оцінку
\begin{equation}\label{eq1K}
M(f_n(\Gamma(S(z_0, r_1), S(z_0, r_2),\,A)))\leqslant
\int\limits_{A(z_0, r_1, r_2)} K^T_{\mu_{f_n}}(z, z_0)\cdot
\eta^2(|z-z_0|)\, dm(z)
\end{equation}
у будь-якому кільці $A=A(z_0, r_1,r_2)=\{z\in {\Bbb C}:
r_1<|z-z_0|<r_2\}$ при довільному $z_0\in D,$ довільних
$0<r_1<r_2<{\rm dist}\,(z_0,
\partial D)$ і кожної вимірної за Лебегом функції $\eta:
(r_1,r_2)\rightarrow [0,\infty]$ такої, що
\begin{equation}\label{eq1L}
\int\limits_{r_1}^{r_2}\eta(r)\, dr \geqslant 1\,.
\end{equation}
Зауважимо, що умову~(\ref{eq1M}) можна переписати в вигляді
$$(f_n)_{\overline{z}}=(f_n)_z\mu_n(z)+
(f_n)_{\overline{z}}\nu_n(z)=(f_n)_z\left(\mu_n(z)+
\frac{(f_n)_{\overline{z}}}{(f_n)_z}\cdot\nu_n(z)\right)\,,$$
тому $\mu_{f_n}(z)=\mu_n(z)+
\frac{(f_n)_{\overline{z}}}{(f_n)_z}+\nu_n(z)$ і, отже,
$$K^T_{\mu_{f_n}}(z, z_0)=\frac{\left|1-\frac{\overline{z-z_0}}{z-z_0}\left(\mu_n(z)+
\frac{(f_n)_{\overline{z}}}{(f_n)_z}\cdot\nu_n(z)\right)\right|^2}{1-|\mu_n(z)+
\frac{(f_n)_{\overline{z}}}{(f_n)_z}\cdot\nu_n(z)|^2}\leqslant
Q^{z_0}_1(z)$$
майже скрізь, бо
$\frac{(f_n)_{\overline{z}}}{(f_n)_z}=e^{i\theta_n(z)}\in D$ при
кожному натуральному $n$ і майже всіх $z,$ де
$\theta_n=\theta_n(z)\in [0, 2\pi).$ Тоді з огляду на
співвідношення~(\ref{eq1E}) і~(\ref{eq1F}) та за лемою~7.6
в~\cite{MRSY} послідовність $f_n$ є одностайно неперервною відносно
хордальної (сферичної) метрики $h$ в $\overline{{\Bbb C}}$ (див.
означення~12.1 у \cite{Va}). За критерієм Арцела-Асколі (див.,
напр., \cite[теорема~20.4]{Va}) існує підпослідовність $f_{n_k}$
послідовності $f_n,$ $k=1,2,\ldots ,$ яка збігається при
$k\rightarrow\infty$ до декого відображення $f$ локально рівномірно
в ${\Bbb C}.$ Отже, за лемою~4.2 в~\cite{RS$_1$} відображення $f$ є
або гомеоморфізмом у ${\Bbb C},$ або сталою в $\overline{{\Bbb C}}.$
Друга ситуація виключена враховуючи умови нормування $f_n(0)=0,$
$|f_n(1)|=1.$ Оскільки $K_{\mu, \nu}(z)\leqslant Q(z)\in L_{\rm
loc}^1(D)$ для майже всіх $z\in D,$ для майже всіх $z\in D$
знайдеться номер $k_0=k_0(z)$ такий, що $\mu_{n_k}(z)=\mu(z),$
$\nu_{n_k}(z)=\nu(z)$ при $n_k\geqslant n_{k_0}(z).$ Отже, для м.в.
$z,$
$$\mu_{n_k}(z)\rightarrow \mu(z)\,,\qquad \nu_{n_k}(z)\rightarrow \nu(z)$$
при $k\rightarrow\infty.$ За твердженням~\ref{pr4.6.2}
$\overline{\partial} f=\mu(z)\partial f+ \nu(z)\partial f,$ тобто,
$f$ -- гомеоморфний розв'язок рівняння~(\ref{eq1G}), причому $f\in
W_{\rm loc}^{1, 1}.$

\medskip
Зауважимо, що за теоремою збіжності розв'язків рівняння Бельтрамі
$f^{\,-1}\in W^{1,2}_{\rm loc}$ (див. \cite[наслідок~2.4]{GRSY}).
Тоді за теоремою Малого-Мартіо $f^{\,-1}$ має $N$-властивість Лузіна
(див., напр., \cite[наслідок~B]{MM}). Нарешті, за теоремою
Пономарьова $J(z, f)\ne 0$ майже скрізь, див. \cite[теорема~1]{Pon}.
Лему доведено. \end{proof} $\Box$

\medskip
{\bf 3. Існування гомеоморфного розв'язку квазілінійного рівняння.}
Для зручності покладемо $\partial f=f_z,$
$\overline{\partial}f=f_{\overline{z}}.$ Наступне твердження
доведено в~\cite[лема~9.1]{GRSY}.

\medskip
\begin{proposition}\label{pr4.6.2}{\sl\,
Нехай $D\subset {\Bbb C},$ і нехай $f_n:D\rightarrow {\Bbb C}$ --
послідовність гомеоморфних розв'язків рівняння $\overline{\partial}
f_n=\mu_n(z)\partial f_n+ \nu_n(z)\overline{\partial f_n}$ класу
$W_{\rm loc}^{1, 1}$ таких, що
$$\frac{1+|\mu_n(z)|+|\nu_n(z)|}{1-|\mu_n(z)|-|\nu_n(z)|}\quad\leqslant\quad
Q(z)\,\in\,L_{loc}^1(D)$$
при всіх $n=1,2,\ldots.$ Якщо $f_n\rightarrow f$ локально рівномірно
в $D$ при $n\rightarrow \infty$ і $f:D\rightarrow {\Bbb C}$ --
гомеоморфізм у $D,$ то $f\in W_{\rm loc}^{1, 1}$ і, крім того,
$\partial f_n$ і $\overline{\partial}f_n$ збігаються слабко в
$L_{loc}^1$ до $\partial f$ і $\overline{\partial}f,$ відповідно.
Якщо $\mu_n\rightarrow\mu$ при $n\rightarrow\infty$ і
$\nu_n\rightarrow\nu$ при $n\rightarrow\infty$ майже скрізь, то
$\overline{\partial} f=\mu(z)\partial f+ \nu(z)\partial f$ майже
скрізь.}
\end{proposition}

\medskip
Нагадаємо, що
$$K_{\mu, \nu}(z, w)=\quad\frac{1+|\mu(z, w)|+|\nu(z, w)|}
{1-|\mu\,(z, w)|-|\nu(z, w)|}$$
і
$$K^T_{\mu, \nu}(z, z_0, w, u, v, \theta)=
\frac{\left|1-\frac{\overline{z-z_0}}{z-z_0}\left(\mu(z, w)+\nu(z,
w)e^{i\theta}\right)\right|^2}{1-|\mu(z, w)+\nu(z,
w)e^{i\theta}|^2}\,.$$
Аналоги наступної леми раніше доводились різними авторами (див.,
напр., \cite[лема~9.1]{GRSY}, \cite[лема~1]{DS$_1$}). Версія, яка
наводиться нижче, відповідає нелінійному рівнянню
Бельтрамі~(\ref{eq1:}).

\medskip
\begin{lemma}\label{lem4.6.1}{\sl\,
Нехай функції $\mu=\mu(z, w)$ і $\nu=\nu(z, w)$ задовольняють умову
Каратеодорі і, крім того, $|\mu(z, w)|+|\nu(z, w)|<1$ при всіх $w\in
{\Bbb C}$ і майже всіх $z\in D.$ Припустимо, що існує функція
$Q:D\rightarrow [1, \infty]$ така, що $K_{\mu, \nu}(z, w)\leqslant
Q(z)\in L_{\rm loc}^1(D)$ для майже всіх $z\in D$ і всіх $w\in {\Bbb
C},$ де функція $K_{\mu, \nu}$ визначена в (\ref{eq1}). Припустимо,
що для будь-якого $z_0\in D$ існує функція $Q^1_{z_0}:D\rightarrow
[0, \infty]$ така, майже всіх $z\in D,$ всіх $w\in {\Bbb C}$ і всіх
$\theta\in [0, 2\pi)$ виконується нерівність
\begin{equation}\label{eq2C}
K^T_{\mu, \nu}(z, z_0, w, \theta)\leqslant Q^1_{z_0}\,.
\end{equation}
Припустимо, що для будь-якого $z_0\in D$ існують
$0<\varepsilon^{\,\prime}_0\leqslant \varepsilon_0< {\rm
dist\,}(z_0,
\partial D),$ і
вимірна за Лебегом функція $\psi_{z_0}:(0, \infty)\rightarrow (0,
\infty)$ такі, що
\begin{equation}\label{eq1A} 0<I(\varepsilon,
\varepsilon_0):=\int\limits_{\varepsilon}^{\varepsilon_0}\psi_{z_0}(t)\,dt
< \infty
\end{equation}
при $\varepsilon \in(0, \varepsilon^{\,\prime}_0),$
$I(\varepsilon)\rightarrow \infty$ при $\varepsilon\rightarrow 0,$
і, при цьому,
\begin{equation}\label{eq10:}
\int\limits_{\varepsilon<|z-z_0|<\varepsilon_0}
Q^1_{z_0}(z)\cdot\psi_{z_0}^{\,2}(|z-z_0|)
 \, dm(z)=o(I^{\,2}(\varepsilon, \varepsilon_0))\,.
\end{equation}
Тоді рівняння (\ref{eq1:}) має регулярний гомеоморфний розв'язок $f$
класу $W_{\rm loc}^{1,1}$ в $D,$ такий що $f^{\,-1}\in W_{\rm
loc}^{\,1,2}(f(D)).$ Більше того, $f$ має гомеоморфне продовження в
${\Bbb C},$ яке є конформним зовні області $D,$ та його можна обрати
таким, що $f(0)=0,$ $f(1)=1.$ }
\end{lemma}

\begin{proof}
Розглянемо наступні послідовності функцій:
\begin{equation}\label{eq12:B} \mu_n(z, w)= \left
\{\begin{array}{rr}
 \mu(z, w),& Q(z)\leqslant n,
\\ 0\ , & Q(z)>n\,,
\end{array} \right.
\end{equation}
\begin{equation}\label{eq12:AB} \nu_n(z, w)= \left
\{\begin{array}{rr}
 \nu(z, w),& Q(z)\leqslant n,
\\ 0\ , & Q(z)>n\,.
\end{array} \right.
\end{equation}
Зауважимо, що $K_{\mu_n, \nu_n}(z,w)\leqslant n$ при майже всіх
$z\in D$ і всіх $w\in{\Bbb C}.$ Отже,
\begin{equation}\label{eq5C}
|\mu_n (z, w)|+|\nu_n (z, w)|\leqslant\frac{n-1}{n+1}<1
\end{equation}
при майже всіх $z\in D$ і майже всіх $w\in {\Bbb C}.$ Розглянемо
рівняння
\begin{equation}\label{eq5D}
\overline{\partial}f_n=\partial f_n\mu_n(z, f_n)+\overline{\partial
f_n}\nu_n(z, f_n)\,.
\end{equation}
Позначимо
$$H_n(z, w, \zeta)=\zeta\mu_n(z, w)+\overline{\zeta}\nu_n(z, w)\,,
\quad \zeta\in {\Bbb C}\,.$$
Тоді для $\zeta, \xi\in {\Bbb C}$ з огляду на~(\ref{eq5C}) та за
нерівністю трикутника ми будемо мати, що
$$|H_n(z, w, \zeta)-H_n(z, w, \xi)|=|\zeta-\xi|\cdot
\biggl|\mu_n(z, w)|+\nu_n(z, w)\cdot
\frac{\overline{\zeta-\xi}}{\zeta-\xi}\biggr|\leqslant$$
$$\leqslant\frac{n-1}{n+1}|\zeta-\xi|\,.$$
В силу умов Каратедорі функція $H_n=H_n(z, w, \zeta)$ є вимірною по
$z\in D$ при кожних фіксованих $w, \zeta\in {\Bbb C}$ і неперервна
по $w, \zeta\in {\Bbb C}$ при майже всіх $z\in D.$ Мит вважаємо
функцію $H_n(z, w, \zeta)$ продовженою нулем зовні області $D.$ Тоді
з огляду на~\cite[теорема~8.2.1]{AIM} рівняння~(\ref{eq5D}) має
гомеоморфний ACL-розв'язок
такий, що $f_n(0)=0,$ $f_n(1)=1.$ Зауважимо, що
рівняння~(\ref{eq5D}) можна записати у вигляді звичайного рівняння
Бельтрамі $\overline{\partial} f_n=\mu^{\,*}_n(z)\partial f_n$ з
одною характеристикою
$\mu^{\,*}_n(z)=\mu_n(z)+\frac{\overline{\partial f_n}}{{\partial
f_n}}\cdot \nu_n(z)\,,$ бо з~(\ref{eq5D}) ми отримаємо, що
$$\overline{\partial} f_n=\left(\mu_n(z)+\frac{\overline{\partial f_n}}{{\partial
f_n}}\cdot \nu_n(z)\right)\partial f_n\,.$$
В цьому випадку, $\bigl|\frac{\overline{\partial f_n}}{{\partial
f_n}}\bigr|=1$ при майже всіх $z\in D,$ тому
$\frac{\overline{\partial f_n}}{{\partial f_n}}=e^{i\theta_n(z)},$
$\theta_n(z)\in [0, 2\pi).$ Обчислюючи за формулами~(\ref{eq20}) і
(\ref{eq2E}) тангенсальну (дотичну) дилатацію відповідну до функції
$\mu^{\,*}_n$ та використовуючи умову~(\ref{eq4B}) разом з
нерівністю трикутника, ми отримаємо, що
\begin{equation}\label{eq8B} K^T_{\mu^{\,*}_n}(z, z_0)=
\frac{\left|1-\frac{\overline{z-z_0}}{z-z_0}
\mu^{\,*}_n(z)\right|^2}{1-|\mu^{\,*}_n(z)|^2}=
\frac{\left|1-\frac{\overline{z-z_0}}{z-z_0}
\left(\mu_n(z)+\frac{\overline{\partial f_n}}{{\partial f_n}}\cdot
\nu_n(z)\right)\right|^2}{1-\bigl|\left(\mu_n(z)+\frac{\overline{\partial
f_n}}{{\partial f_n}}\cdot \nu_n(z)\right)\bigr|^2}\leqslant
Q_{z_0}(z)
\end{equation}
при майже всіх $z\in D.$

\medskip
З іншого боку, з огляду на теорему~11.1 в~\cite{MRSY} та
оцінку~(\ref{eq8B}) кожне $f_n$ задовольняє нерівність
\begin{equation}\label{eq1N}
M(f_n(\Gamma(S(z_0, r_1), S(z_0, r_2),\,A)))\leqslant
\int\limits_{A(z_0, r_1, r_2)} Q_{z_0}(z)\cdot \eta^2(|z-z_0|)\,
dm(z)
\end{equation}
у будь-якому кільці $A=A(z_0, r_1,r_2)=\{z\in {\Bbb C}:
r_1<|z-z_0|<r_2\}$ при довільному $z_0\in D,$ довільних
$0<r_1<r_2<{\rm dist}\,(z_0,
\partial D)$ і кожної вимірної за Лебегом функції $\eta:
(r_1,r_2)\rightarrow [0,\infty]$ такої, що
\begin{equation}\label{eq1O}
\int\limits_{r_1}^{r_2}\eta(r)\, dr \geqslant 1\,.
\end{equation}
Тоді з огляду на співвідношення~(\ref{eq1A}) і~(\ref{eq10:}) та за
лемою~7.6 в~\cite{MRSY} послідовність $f_n$ є одностайно неперервною
відносно хордальної (сферичної) метрики $h$ в $\overline{{\Bbb C}}$
(див. означення~12.1 у \cite{Va}). За критерієм Арцела-Асколі (див.,
напр., \cite[теорема~20.4]{Va}) існує підпослідовність $f_{n_k}$
послідовності $f_n,$ $k=1,2,\ldots ,$ яка збігається при
$k\rightarrow\infty$ до декого відображення $f$ локально рівномірно
в ${\Bbb C}.$ Отже, за лемою~4.2 в~\cite{RS$_1$} відображення $f$ є
або гомеоморфізмом у ${\Bbb C},$ або сталою в $\overline{{\Bbb C}}.$
Друга ситуація виключена враховуючи умови нормування $f_n(0)=0,$
$f_n(1)=1.$ За твердженням~\ref{pr4.6.2} $f\in W_{\rm loc}^{1,
1}(D).$

\medskip
Нарешті, з огляду на умову $K_{\mu, \nu}(z, w)\leqslant Q(z)\in
L_{\rm loc}^1(D)$ для майже всіх $z_0\in D$ знайдеться номер
$n_0=n_0(z_0)\in {\Bbb N}$ такий, що $\mu_n(z, w)=\mu(z, w)$ при
всіх $n\geqslant n_0.$ Тоді з огляду на умови Каратеодорі
для м.в. $z\in D$ і всіх $n_k\geqslant n_0$
$$\mu_{n_k}(z)= \mu_{n_k}(z, f_{n_k}(z))=\mu(z, f_{n_k}(z))\rightarrow \mu(z,
f(z))\,,$$
$$\nu_{n_k}(z)= \nu_{n_k}(z, f_{n_k}(z))=\nu(z, f_{n_k}(z))\rightarrow \nu(z,
f(z))$$
при $k\rightarrow\infty,$ бо функції $\mu$ і $\nu$ неперервні по
другому аргументу за умовою. Тоді за твердженням~\ref{pr4.6.2}
$\overline{\partial} f=\mu(z, f)\partial f+ \nu(z, f)\partial f,$
тобто, $f$ -- гомеоморфний розв'язок рівняння~(\ref{eq1:}), причому
$f\in W_{\rm loc}^{1, 1}.$

\medskip
Зауважимо, що за теоремою збіжності розв'язків рівняння Бельтрамі
$f^{\,-1}\in W^{1,2}_{\rm loc}$ (див. \cite[наслідок~2.4]{GRSY}).
Тоді за теоремою Малого-Мартіо $f^{\,-1}$ має $N$-властивість Лузіна
(див., напр., \cite[наслідок~B]{MM}). Нарешті, за теоремою
Пономарьова $J(z, f)\ne 0$ майже скрізь, див. \cite[теорема~1]{Pon}.
Рівності $f_n(0)=0$ і $f_n(1)=1$ миттєво тягнуть за собою, що
$f(0)=0,$ $f(1)=1.$ Лема доведена. \end{proof} $\Box$

\medskip
{\it Доведення теореми~\ref{th4.6.1}} випливає з леми~\ref{lem4.6.1}
і \cite[лема~1.3]{Sev$_2$}.

\medskip
{\bf 4. Існування неперервного розв'язку.} Окремо сформулюємо умови,
за яких рівняння~(\ref{eq1:}) має лише неперервний розв'язок, див.
\cite{DS$_1$}--\cite{DS$_2$} і \cite{SevSkv}. Нехай $J(z, f)\ne 0$ і
нехай відображення $f$ має частинні похідні $f_z$ і
$f_{\overline{z}}$ у точці $z.$ Тоді {\it максимальною дилатацією
відображення} $f$ в точці $z$ будемо називати наступну функцію:
\begin{equation}\label{eq1D}
K_{\mu_f}(z)=\frac{|f_z|+|f_{\overline{z}}|}{|f_z|-|f_{\overline{z}}|}\,.
\end{equation}
Покладемо $K_{\mu_f}(z)=1$ в точках $z,$ де
$|f_z|+|f_{\overline{z}}|=0$ та $K_{\mu_f}(z)=\infty$ у точках $z,$
де $|f_z|+|f_{\overline{z}}|\ne 0,$ але $J(z, f)=0.$ Визначимо також
{\it внутрішню дилатацію порядку $p\geqslant 1 $} відображення $f$
за допомогою співвідношення
\begin{equation}\label{eq17}
K_{I, p}(z,
f)=\frac{{|f_z|}^2-{|f_{\overline{z}}|}^2}{{(|f_z|-|f_{\overline{z}}|)}^p}
\end{equation}
де $J(z, f)\ne 0.$ Як і у (\ref{eq1D}), покладемо $K_{I, p}(z)=1,$
якщо $|f_z|+|f_{\overline{z}}|=0,$ і $K_{I, p}(z)=\infty$ у точках,
де $J(z, f)=0,$ проте $|f_z|+|f_{\overline{z}}|\ne 0.$ Зауважимо, що
$K_{I, 2}(z)=K_{\mu}(z).$ Покладемо $\Vert
f^{\,\prime}(z)\Vert=|f_z|+|f_{\overline{z}}|.$  Нагадаємо, що
гомеоморфізм $f:D\rightarrow {\Bbb C}$ називається {\it
квазіконформним,} якщо $f\in W_{\rm loc}^{1, 2}(D)$ і, крім того,
існує стала $K\geqslant 1$ така, що $\Vert
f^{\,\prime}(z)\Vert^2\leqslant K\cdot |J(z, f)|$ майже скрізь.

\medskip
Справедлива наступна лема, доведена в~\cite[лема~5]{DS$_2$}, див.
також~\cite[лема~2]{DS$_1$}, \cite[теорема~9.1]{GRSY} і
\cite[лема~5.1]{SevSkv}.

\medskip
\begin{lemma}\label{lem1}
{\sl\, Нехай $1<p\leqslant 2,$ нехай $\mu:D\rightarrow {\Bbb D}$ --
вимірна за Лебегом функція, і нехай $f_k,$ $k=1,2,\ldots $ --
послідовність гомеоморфізмів, що зберігають орієнтацію, області $D$
на себе, які належать класу $W_{\rm loc}^{1, 2}(D)$ і задовольняють
рівняння
\begin{equation}\label{eq1B}
\overline{\partial}f_n=\partial f_n\mu_n(z)+\overline{\partial
f_n}\nu_n(z)\,,
\end{equation}
де $\mu_n,$ $\nu_n$ -- вимірні за Лебегом функції, які задовольняють
нерівність $|\nu_n(z)|+|\mu_n(z)|<1$ майже скрізь. Припустимо, що
$f_n$ збігається локально рівномірно в $D$ до відображення
$f:D\rightarrow {\Bbb C},$ а послідовності $\mu_n(z)$ та $\nu_n(z)$
збігаються до $\mu(z)$ і $\nu(z),$ відповідно, при
$n\rightarrow\infty$ майже скрізь. Нехай також обернені відображення
$g_n:=f_n^{\,-1}$ належать класу $W_{\rm loc}^{1, 2}(f_n(D)),$ при
цьому, при майже всіх $w\in f_n(D)$
$$\int\limits_{f_n(D)}K_{I, p}(w, g_k)\,dm(w)\leqslant M$$
для деякого $M>0$ і кожного $n=1,2,\ldots .$

Тоді $f\in W_{\rm loc}^{1, p}(D)$ і $\mu,$ $\nu$  -- комплексні
характеристики відображення $f,$ тобто,
$\overline{\partial}f=\partial f\mu(z)+\overline{\partial f}\nu(z)$
при майже всіх $z\in D.$
 }
\end{lemma}

\medskip
Нехай $\mu=\mu(z, w):D\times {\Bbb C}\rightarrow {\Bbb D}$ і
$\nu=\nu(z, w): D\times {\Bbb C}\rightarrow {\Bbb D}$ -- задані
функції. Зафіксуємо $n\geqslant 1$ і покладемо
\begin{equation}\label{eq12:} \mu_n(z, w)= \left
\{\begin{array}{rr}
 \mu(z, w),&  K_{\mu, \nu}\leqslant n,
\\ 0\ , & K_{\mu, \nu}> n\,,
\end{array} \right.
\end{equation}
і
\begin{equation}\label{eq12:A} \nu_n(z, w)= \left
\{\begin{array}{rr}
 \nu(z, w),& K_{\mu, \nu}\leqslant n,
\\ 0\ , & K_{\mu, \nu}> n\,.
\end{array} \right.
\end{equation}
Нехай функції $\mu=\mu(z, w):D\times{{\Bbb C}}\rightarrow {\Bbb D}$
і $\nu=\nu(z, w):D\times{{\Bbb C}}\rightarrow {\Bbb D}$
задовольняють умову Каратеодорі. Тоді
$$K_{\mu_n, \nu_n}(z, w)=\quad\frac{1+|\mu_n(z, w)|+|\nu_n(z, w)|}
{1-|\mu_n\,(z, w)|-|\nu_n(z, w)|}\leqslant n$$
для всіх $z\in D$ і $w\in {\Bbb C}.$ Тоді за наслідком~\ref{cor1}
рівняння
$$f_{\overline{z}}=\mu_n(z, f(z))\cdot
f_z+\nu_n(z, f(z))\cdot \overline{f_z}$$
має регулярний гомеоморфний розв'язок $f_n$ класу $W_{\rm
loc}^{1,1}$ в $D,$ такий що $f_n^{\,-1}\in W_{\rm
loc}^{\,1,2}(f_n(D)).$ За цим же наслідком можна продовжити $f_n$ на
все ${\Bbb C}$ і обрати його таким, що $f_n(0)=0,$ $f_n(1)=1.$

Нехай $f_n$ -- (будь-який) такий розв'язок рівняння~(\ref{eq1:}).
Оскільки рівність
$$(f_n){\overline{z}}=\mu_n(z, f_n(z))\cdot
(f_n)_z+\mu_n(z, f_n(z))\cdot \overline{(f_n)_z}$$
можна записати в вигляді
\begin{equation}\label{eq8C}
\overline{\partial} f_n=\left(\mu_n+\frac{\overline{\partial
f_n}}{{\partial f_n}}\cdot \nu_n\right)\partial f_n\,,
\end{equation}
де $\mu_n=\mu_n(z, f_n(z))$ і $\nu_n=\nu_n(z, f_n(z)).$ З
рівності~(\ref{eq8C}) випливає, що відображення $f_n$ при кожному
$n\in{\Bbb N}$ задовольняє звичайне рівняння Бельтрамі
$f_{\overline{z}}=\mu^{\,*}(z)f_z,$ де
$\mu^{\,*}(z)=\mu_n+\frac{\overline{\partial f_n}}{{\partial
f_n}}\cdot \nu_n.$ Оцінюючи максимальну дилатацію $K_{\mu^{\,*}}(z)$
цього рівняння за нерівністю трикутника, ми отримаємо, що
\begin{equation}\label{eq9B}
K_{\mu^{\,*}}(z)\leqslant\frac{1+|\mu_n|+|\nu_n|}{1-|\mu_n|-|\nu_n|}\leqslant
n\,.
\end{equation}
З оцінки~(\ref{eq9B}) випливає, що відображення $f_n$ є
квазіконформним. Тоді й відображення $g_n=f^{\,-1}_n$ є
квазіконформним; зокрема, воно є диференційовним майже скрізь. Нехай
$K^T_{\mu_{g_n}}(w, w_0)$ -- дотична дилатація оберненого
відображення $g_n$ у точці $w\in {\Bbb C}$ відносно точки $w_0\in
{\Bbb C}$ тобто,
\begin{equation}\label{eq3D}
K^T_{\mu_{g_n}}(w, w_0)=
\frac{\left|1-\frac{\overline{w-w_0}}{w-w_0}\mu_{g_n}(w)\right|^2}{1-|\mu_{g_n}(w)|^2}\,.
\end{equation}
Визначимо також {\it внутрішню дилатацію порядку $p$ відображення
$g_n$ в точці $w$} за допомогою рівності
\begin{equation}\label{eq18}
K_{I, p}(w,
g_n)=\frac{{|(g_n)_w|}^2-{|(g_n)_{\overline{w}}|}^2}
{{(|(g_n)_w|-|(g_n)_{\overline{w}}|)}^p}\,.
\end{equation}

\medskip
Виконується наступне твердження (для лінійних рівнянь з одною і
двома характеристиками в одиничному крузі і максимальною дилатацією
замість дотичної див. також \cite{Sev$_1$}, \cite{SevSkv} і
\cite{DS$_1$}--\cite{DS$_2$}).

\medskip
\begin{theorem}\label{th1A}{\sl\, Нехай $\mu,$ $\nu,$ $\mu_n,$
$\nu_n,$ $f_n$ і $g_n$ такі, як визначено вище. Нехай $Q, Q_0:{\Bbb
C}\rightarrow[0, \infty]$ -- вимірні за Лебегом функції. Припустимо,
що $K_{\mu, \nu}(z, w)\leqslant Q_0(z)<\infty$ при всіх $w\in {\Bbb
C}$ і майже всіх $z\in D.$ Припустимо, крім того, що виконуються
наступні умови:

\medskip
1) для кожних $0<r_1<r_2<1$ і $y_0\in {\Bbb C}$ існує множина
$E\subset[r_1, r_2]$ додатної лебегової міри така, що функція $Q$ є
інтегровною по колах $S(y_0, r)$ для кожного $r\in E;$

\medskip
2) знайдеться число $1<p\leqslant 2$ і стала $M>0$ такі, що
\begin{equation}\label{eq10B}
\int\limits_{f_n(D)}K_{I, p}(w, g_n)\,dm(w)\leqslant M
\end{equation}
для всіх $n=1,2,\ldots ,$ де $K_{I, p}(w, g_n)$ визначено
у~(\ref{eq18});

\medskip
3) Нерівність
\begin{equation}\label{eq10C}
K^T_{\mu_{g_n}}(w, w_0)\leqslant Q(w)
\end{equation}
виконується для майже всіх $w\in f_n(D)$ і всіх $w_0\in f_n(D),$ де
$K_{\mu_{g_n}}$ визначено в~(\ref{eq3D}). Тоді рівняння~(\ref{eq1:})
має неперервний $W_{\rm loc}^{1, p}(D)$-розв'язок $f$ в $D.$}
\end{theorem}

\medskip
\begin{corollary}\label{cor5}
{\sl\, Зокрема, твердження теореми~\ref{th1A} виконується, якщо в
цій теоремі ми відмовляємося від умови~1), вимагаємо умову~3), а
умову 2) заміняємо наступною: $Q\in L^1({\Bbb C}).$ В цьому випадку,
розв'язок $f$ рівняння~(\ref{eq1:}) може бути обраним таким, що
рівність
\begin{equation}\label{eq11A}
|f(x)-f(y)|\leqslant\frac{C\cdot (\Vert
Q\Vert_1)^{1/2}}{\log^{1/2}\left(1+\frac{r_0}{2|x-y|}\right)}
\end{equation}
виконується для довільного компакту $K\subset D$ і всяких $x, y\in
K,$ де $\Vert Q\Vert_1$ позначає $L^1$-норму функції $Q$ в ${\Bbb
C},$ $C>0$ деяка стала і $r_0=d(K,
\partial D).$ Якщо додатково $Q(z)\in FMO({\Bbb C}),$ або
виконано умову~(\ref{eq15:}),
то відображення $f$ можна обрати гомеоморфізмом.}
\end{corollary}

\medskip
{\it Доведення теореми~\ref{th1A}.}  Нехай $M$ означає модуль сім'ї
кривих (див., напр.,~\cite[розд.~6]{Va}), $w_0\in f_n(D),$
$0<r_1<r_2<\sup\limits_{w\in f_n(D)}|w-w_0|.$ Тоді з огляду
на~\cite[теорема~4.2]{RSSY} і умову~(\ref{eq10C})
$$M(g_n(\Sigma_{r_1, r_2}))\geqslant \int\limits_{r_1}^{r_2}
\frac{dr}{\Vert Q\Vert_1(r)}\,,$$
де $\Sigma_{r_1, r_2}$ позначає сім'ю всіх перетинів кіл $S(y_0,
r)\cap f_n(D),$ $r\in (r_1, r_2),$ і $\Vert
Q\Vert_1(r)=\int\limits_{S(y_0, r)\cap
f_n(D)}Q(y)\,\,d\mathcal{H}^{1}(y).$ Нехай $C_1$ і $C_2$ -- компакти
у $D,$ такі що
$$
C_1\subset S(w_0, r_1)\quad\mbox{і}\quad C_2\subset S(w_0, r_2)\,.
$$
Застосовуючи результати Цимера про зв'язок між модулем сімей
з'єднуючих кривих та відповідних розділяючих поверхонь
(див.~\cite[теорема~3.13]{Zi$_1$} при $p=n$ і \cite[с.~50]{Zi$_2$}
при $1<p<\infty$), а також результати Хессе-Шлик про зв'язок модуля
сімей кривих з ємністю конденсаторів (див.~\cite[теорема~5.5]{Hes}
і~\cite[теорема~1]{Shl}), ми отримаємо що
\begin{equation}\label{eq5}
M(g_n(\Gamma(C_1, C_2, f_n(D))))\leqslant
\frac{2\pi}{\int\limits_{r_1}^{r_2}
\frac{dr}{rq_{w_0}(r)}}<\infty\,,
\end{equation}
де $q_{w_0}(r)$ визначено співвідношенням~(\ref{eq2D}). Тут ми
скористалися умовою~1) теореми, бо нерівність $q_{w_0}(r)<\infty$
тягне за собою скінченність правої частини у співвідношенні вище, а
також умовою~3). Візьмемо тепер зростаючу послідовність компактів
$C_1^m$ і $C_2^m,$ $m=1,2,\ldots ,$ які вичерпують $S(w_0, r_1)\cap
f_n(D)$ and $S(w_0, r_2)\cap f_n(D)$ і перейдемо до границі
$m\rightarrow\infty$ у співвідношенні
$$
M(g_n(\Gamma(C_1^m, C_1^m, f_n(D))))\leqslant
\frac{2\pi}{\int\limits_{r_1}^{r_2}
\frac{dr}{rq_{w_0}(r)}}<\infty\,.
$$
З огляду на~\cite[теорема~A.7]{MRSY},
$$M(g_n(\Gamma(S(w_0, r_1), S(w_0, r_2), f_n(D))))\leqslant
\frac{2\pi}{\int\limits_{r_1}^{r_2}
\frac{dr}{rq_{w_0}(r)}}<\infty\,.$$
  Останню нерівність можна записати в дещо іншій
формі:
$$M(\Gamma_{f_n}(w_0, r_1, r_2))\leqslant \frac{2\pi}{\int\limits_{r_1}^{r_2}
\frac{dr}{rq_{w_0}(r)}}<\infty\,.$$
Тоді з огляду на пропозицію~3 і зауваження~4 у~\cite{DS$_2$} сім'я
$f_n$ одностайно неперервна в $D.$ Отже, з огляду на теорему
Арцела-Асколі $f_n$ є нормальною сім'єю відображень
(див.~\cite[теорема~20.4]{Va}). Іншими словами, знайдеться
підпослідовність $f_{n_l},$ котра збігається до деякого відображення
$f:D\rightarrow \overline{D}$ локально рівномірно в $D.$ За
означенням $f_n,$
$$\overline{\partial} f_n=\mu_n(z, f_n(z))\partial f_n+
\nu_n(z, f_n(z))\overline{\partial f_n}\,.$$
Оскільки $K_{\mu, \nu}(z, w)\leqslant Q_0(z)<\infty$ при всіх $w\in
{\Bbb C}$ і майже всіх $z\in D.$ Отже, для майже всіх $z\in D$
знайдеться номер $l_0=l_0(z)$ такий, що $\mu_{n_l}(z,w)=\mu(z,w),$
$\nu_{n_l}(z,w)=\nu(z,w)$ при $l\geqslant l_0$ і всіх $w\in{\Bbb
C}.$ З огляду на те, що функції $\mu$ і $\nu$ задовольняють умову
Каратеодорі,
$$\mu_{n_l}(z, f_{n_l}(z))\rightarrow \mu(z, f(z))\,,$$
$$\nu_{n_l}(z, f_{n_l}(z))\rightarrow \nu(z, f(z))$$
при $l\rightarrow\infty.$ Тоді за умовою~2) лемою~\ref{lem1}
відображення $f$ належить класу $W_{\rm loc}^{1, p}(D)$ і є
розв'язком вихідного рівняння Бельтрамі~(\ref{eq1:}).~$\Box$

\medskip
{\it Доведення наслідку~\ref{cor5}.} За теоремою Фубіні умова $Q\in
L^1(D)$ тягне за собою вимірність інтегралів $\int\limits_{S(x_0,
r)\cap D}\,Q(x)\,d\mathcal{H}^1(x)$ як функцій $r$ і їх скінченність
майже скрізь при $0<r<\infty$ (див., напр.,
\cite[теорема~8.1.III]{Sa}). Отже, умова~(\ref{eq10B}) виконується
при $p=2.$ Тоді існування розв'язку рівняння~(\ref{eq1:}) і його
приналежність класу~$W_{\rm loc}^{1, 2}(D)$ випливають з
теореми~\ref{th1A}.

\medskip
З огляду на~\cite[теорема~1.2]{SevSkv} (див. також
\cite[теорема~1]{SSD}) на довільному компакті $K\subset D$
виконується нерівність
$$|f_n(x)-f_n(y)|\leqslant\frac{C\cdot (\Vert
Q\Vert_1)^{1/2}}{\log^{1/2}\left(1+\frac{r_0}{|x-y|}\right)}\quad\forall\,\,x,y\in
K\,,$$
де $\Vert Q\Vert_1$ -- норма $Q$ в $L^1(D),$ $C$ -- деяка стала і
$r_0=d(K,
\partial D).$ Переходячи тут до границі при
$n\rightarrow\infty,$ ми отримаємо~(\ref{eq11A}).

\medskip
Припустимо тепер, що~$Q\in FMO(D),$ або виконується
співвідношення~(\ref{eq15:}). Оскільки $f_n$ продовжуються на все
${\Bbb C}$ і є квазіконформними, то вони мають неперервне
продовження у точку $\infty$ до квазіконфомного відображення
$f_n:\overline{\Bbb C}=\overline{\Bbb C}$ (див., напр.,
\cite[теорема~17.3]{Va}). Оскільки $f_n(\overline{\Bbb C})$ є
відкритою множиною як образ відкритої множини $\overline{\Bbb C}$
при гомеоморфізмі $f_n,$ і замкненою як образ компакту
$\overline{\Bbb C}$ при неперервному відображенні $f_n,$ то
$f_n(\overline{\Bbb C})=\overline{\Bbb C}.$ Звідси випливає, що
$f_n({\Bbb C})={\Bbb C}.$ Тому всі відображення $g_n=f^{\,-1}_n$
визначені у ${\Bbb C}.$

Крім того, з огляду на умови нормування $g_n=0$ і $g_n(1)=1$
послідовність $g_n$ формує одностайно неперервну сім'ю (див., напр.,
~\cite[теореми~6.1 і 6.5]{RS}). За теоремою Арцела-Асколі $g_n$ є
нормальною сім'єю (див.~\cite[теорема~20.4]{Va}), тобто, існує
підпослідовність $g_{n_l},$ яка збігається до деякого
$g:D\rightarrow \overline{D}$ локально рівномірно у деякій області
$D.$ Оскільки знову $g_{n_l}(0)=0$ і $g_{n_l}(1)=1$ при всіх
$l=1,2,\ldots ,$ то в силу~\cite[теорема~4.1]{RS$_1$} відображення
$g$ є гомеоморфізмом в $D,$ причому $f_{n_l}\rightarrow f=g^{\,-1}$
при $l\rightarrow\infty$ локально рівномірно в $D$
(див.~\cite[лема~3.1]{RS$_1$}). Оскільки за умовами Каратеодорі
$\mu_n(z, f_n(z))\rightarrow \mu(z, f)$ і $\nu_n(z,
f_n(z))\rightarrow \nu(z, f(z))$ при $n\rightarrow\infty$ і при
майже всіх $z\in D,$ за лемою~\ref{lem1} відображення $f$ належить
класу $W_{\rm loc}^{1, p}(D)$ і задовольняє
рівняння~(\ref{eq1:}).~$\Box$

\medskip
{\bf 4. Приклад.}  Будемо вважати $\nu\equiv 0.$ В одиничному крузі
$D={\Bbb D}=\{z\in {\Bbb C}: |z|<1\}$ розглянемо функцію $\mu:{\Bbb
D}\times {\Bbb C}\rightarrow {\Bbb D}$ визначену наступним шляхом:
$$\mu(z, w)=e^{i\theta}\frac{1-r-|w|}{1+r+|w|}\,,$$
де $z=re^{i\theta},$ $0\leqslant r<1,$ $\theta\in [0, 2\pi).$
Зауважимо, що
\begin{equation}\label{eq9A}
K_{\mu}(z, w)=\frac{1+|\mu(z, w)|}{1-|\mu(z,
w)|}=\frac{1}{r+w}\leqslant \frac{1}{r}\,.
\end{equation}
Функція $Q(z):=\frac{1}{r}$ інтегровна в ${\Bbb D}.$ Дійсно, з
огляду на теорему Фубіні,
$$\int\limits_{\Bbb D}Q(z)\,dm(z)=\int\limits_{0}^1\left(\int\limits_{S(0, r)}
\frac{1}{r}\,\,dt\right)\,dr=2\pi<\infty\,.$$
У той же час, умова типу~(\ref{eq15:}) не виконується для функції
$Q,$ адже при $0<\delta(0)<1$
\begin{equation}\label{eq9C}
\int\limits_{0}^{\delta(0)}\frac{dr}{rq_{0}(r)}=
\int\limits_{0}^{\delta(0)}\frac{dr}{r\cdot
\frac{1}{r}}=\delta(0)<\infty\,,
\end{equation}
де $q_{0}(r)=\frac{1}{2\pi}Q(z_0+re^{it})\,dt.$ Це означає, що умови
наслідку~\ref{cor3}, які передбачають виконання
співвідношення~(\ref{eq9C}) для максимальної дилатації $K_{\mu}(z,
w),$ не виконуються.

Проте, покажемо, що відповідна умова~(\ref{eq9C}) виконана по
відношенню до дотичної дилатації. Справді, для $z_0=0$
$$K^T_{\mu(z, 0, w)}=\frac{\left|1-\frac{\overline{z}}{z}\left(\mu(z,
w)\right)\right|^2}{1-|\mu(z, w)|^2}=$$$$ =\frac{\left(1-\left(
\frac{1-r-|w|}{1+r+|w|}\right)\right)^2}{1-\biggl(
\frac{1-r-|w|}{1+r+|w|}\biggr)^2}=\frac{1-\left(
\frac{1-r-|w|}{1+r+|w|}\right)}{1+ \frac{1-r-|w|}{1+r+|w|}}<1$$
для всіх $r\in (0, 1)$ і всіх $w\in {\Bbb C}.$ Якщо покласти
$Q^1_{0}(z)\equiv 1,$ то можна зауважити, що функція $Q^1_{0}(z)$
задовольняє умову~(\ref{eq15:}) у точці $z_0=0.$ Виконання цієї
умови в інших точках $z_0$ одиничного круга ${\Bbb D}$ є очевидним з
огляду на те, що функція $K_{\mu}(z, w)$ задовольняє
умову~(\ref{eq9A}) і функція $Q(z):=\frac{1}{r}$ є локально
обмеженою в ${\Bbb D}.$

\medskip
Остаточно, за теоремою~\ref{th4.6.1} рівняння
$$f_{\overline{z}}=e^{i\theta}\frac{1-r-|f(z)|}{1+r+|f(z)|}
\cdot f_z$$
має гомеоморфний $W_{\rm loc}^{1,1}$-розв'язок у ${\Bbb D}.$

КОНТАКТНА ІНФОРМАЦІЯ

\medskip
\noindent{{\bf Євген Олександрович Севостьянов} \\
{\bf 1.} Житомирський державний університет ім.\ І.~Франко\\
вул. Велика Бердичівська, 40 \\
м.~Житомир, Україна, 10 008 \\
{\bf 2.} Інститут прикладної математики і механіки
НАН України, \\
вул.~Добровольського, 1 \\
м.~Слов'янськ, Україна, 84 100\\
e-mail: esevostyanov2009@gmail.com}

\medskip
\noindent{{\bf Валерій Андрійович Таргонський} \\
Житомирський державний університет ім.\ І.~Франко\\
вул. Велика Бердичівська, 40 \\
м.~Житомир, Україна, 10 008 \\
e-mail: w.targonsk@gmail.com }

\medskip
\noindent{{\bf Ількевич Наталія Сергіївна} \\
Житомирський державний університет ім.\ І.~Франка\\
вул. Велика Бердичівська, 40 \\
м.~Житомир, Україна, 10 008 \\
e-mail: ilkevych1980@gmail.com }

\end{document}